\documentclass[english]{amsart}

\usepackage[all,cmtip]{xy}
\usepackage{todonotes}

\usepackage{amsmath,amssymb,amscd}

\usepackage{babel}
\usepackage{amstext}
\usepackage{amsmath}
\usepackage{amsfonts}
\usepackage{latexsym}
\usepackage{ifthen}
\xyoption{all}
\usepackage{enumitem}
\setlist[enumerate]{label=(\thethm.\arabic*), before={\setcounter{enumi}{\value{equation}}}, after={\setcounter{equation}{\value{enumi}}}}

\newcommand{\Q}{\mathbb{Q}}

\newcommand{\Z}{\mathbb{Z}}
\newcommand{\N}{\mathbb{N}}

\renewcommand{\epsilon}{\varepsilon}

\renewcommand{\leq}{\leqslant}
\newcommand{\holom}[3]{\ensuremath{#1:#2  \rightarrow #3}}
\renewcommand{\geq}{\geqslant}

\newcommand{\Res}{\mathrm{Res}}

\newcommand\sO{{\mathcal O}}

\newcommand{\dbar}{\bar \partial}

\newtheorem{thm}{Theorem}[section]
\newtheorem{lemma}[thm]{Lemma}
\newtheorem{proposition}[thm]{Proposition}
\newtheorem{question}[thm]{Question}

\newtheorem{cor}[thm]{Corollary}
\theoremstyle{remark}

\newtheorem{remark}[thm]{Remark}

\DeclareMathOperator*{\nons}{nons}

\newcommand\sF{{\mathcal F}}
\newcommand\sA{{\mathcal A}}

\newenvironment{remark*}{{\em Remark.}}{}

\numberwithin{equation}{section}

\title{Frobenius integrability of certain $p$-forms on singular spaces} 
\date{October 6, 2022}

\author{}

\author{Junyan Cao}
\author{Andreas H\"oring}

\address{Junyan Cao, Universit\'e C\^ote d'Azur, CNRS, LJAD, France, Institut universitaire de France}
\email{junyan.cao@univ-cotedazur.fr}

\address{Andreas H\"oring, Universit\'e C\^ote d'Azur, CNRS, LJAD, France, Institut universitaire de France}
\email{Andreas.Hoering@univ-cotedazur.fr}


\begin{document}

\begin{abstract}
Demailly proved that on a smooth compact K\"ahler manifold the distribution defined by a holomorphic $p$-form with values in an anti-pseudoeffective line bundle is always integrable. We generalise his result to compact K\"ahler spaces with klt singularities.
\end{abstract}

\maketitle

\section{Introduction}

Let $X$ be a compact K\"ahler manifold, and let $u \in H^0(X, \Omega_X^p)$ be a holomorphic $p$-form on $X$. As a consequence of the K\"ahler identity for the Laplacians $\Delta_d = 2 \Delta_{\partial}$ one obtains that the holomorphic form is $d$-closed, i.e. $d u=0$. Twenty years ago Jean-Pierre Demailly used a very clever ``integration by parts'' to generalise this statement to forms with values in certain line bundles:

\begin{thm} \cite[Main thm]{Dem02}  \label{theorem-demailly}
Let $X$ be a compact K\"ahler manifold. Let $(L,h)$ be a holomorphic line bundle on $X$ where $h$ is a possibly singular metric such that  $i\Theta_h (L)\geq 0$ on $X$ in the sense of currents. 
Let 
$$
u \in H^0(X, \Omega_X^p \otimes L^\star)
$$
be a non-zero holomorphic section, and let 
$S_u \subset T_X$ be the saturated subsheaf given by vector fields $\xi$ such that the contraction $i_\xi u$ vanishes.

Then one has $D' _{h^\star} u =0$. Hence $S_u$ is integrable, i.e. it defines a (possibly singular) foliation on $X$, and $(L,h)$ has flat curvature along the leaves.
\end{thm}

Demailly's main motivation for this result was to prove that if a compact K\"ahler manifold admits a contact structure, then the canonical bundle $K_X$ is never pseudoeffective \cite[Cor.2]{Dem02}. Moreover Theorem \ref{theorem-demailly} has turned out to be a very efficient tool for the study of foliations with vanishing first Chern class \cite{PT13, LPT18, GKP21}. In view of the increased interest in foliations on singular spaces  (cf. e.g. \cite{CS21,Dru21}) it seems worthwhile to look at Demailly's arguments in this setting. Our main result is:

\begin{thm} \label{theorem-main}
Let $Y$ be a normal compact K\"ahler space with klt singularities.
Let $\sA$ be a rank one reflexive sheaf such that the reflexive power $\sA^{[m]}$ is locally free and pseudoeffective for some $m \in \N$. Let 
$$
u \in H^0(Y, (\Omega_Y^p \otimes \sA^\star)^{\star\star})
$$
be a non-zero holomorphic section. Let $S_u \subset T_Y$ be the saturated subsheaf given by vector field $\xi$ such that the contraction $i_\xi u$ vanishes.
Then $S_u$ is integrable, i.e. it defines a (possibly singular) foliation on $Y$.
\end{thm}

For applications in foliation theory it is interesting to verify if $\sA$ has flat curvature
along the leaves of $S_u$. Since $\sA$ is not locally free the precise formulation  would be a bit awkward, but flatness holds for the corresponding line bundle $(L,h)$ 
on a resolution of singularities (cf. Proposition \ref{l2version} and Proposition \ref{pless2}).

Our basic strategy is similar to the proof of Theorem \ref{theorem-demailly}, except that we have to carry out the computation on a resolution of singularities $\holom{\pi}{X}{Y}$. If $\sA$ is not locally free this leads to some well-known difficulties, for example the saturation of $\pi^\star \sA$ in $\Omega^p_{X}$ is not always pseudoeffective \cite{GKP12, Ou14}. Therefore we consider forms with logarithmic poles along the exceptional divisor $E$ of the resolution $\pi$, in particular we obtain that the saturation in $\Omega^p_{X}(\log E)$ is pseudoeffective, cf. Corollary \ref{corollary-pseff}. 

This leads us to the following problem:

\begin{question} \label{lccase}
Let $(X,\omega_X)$ be a compact K\"ahler manifold, and let $E=\sum E_i$ be a snc divisor. Let $(L,h)$ be a holomorphic line bundle on $X$ where $h$ is a possibly singular metric such that  $i\Theta_h (L)\geq 0$ on $X$ in the sense of currents. Let $(L^\star, h^\star)$ be the dual metric.  

Let $u\in H^0 (X, \Omega_X ^p (\log E) \otimes L^\star)$. Can we prove that $D' _{h^\star} u =0$ on $X\setminus E$, where $D' _{h^\star}$ is the connection with respect to $h^\star$ ?	
\end{question}

If $p=1$, the problem is totally solved in \cite[Thm 5]{Tou16}\footnote{We thank St\'ephane Druel and Daniel Greb for bringing this reference to our attention.}. It is still open when $p\geq 2$.
We give a positive answer to this question when the metric $h$ is smooth (Proposition \ref{proposition-smooth-metric}). Our main technical result (Proposition \ref{l2version}) gives a positive answer making an assumption on the singularity of $h$ along certain irreducible components $E_i$. This integrability condition can be verified for a resolution of singularities $X \rightarrow Y$ of a klt space. 
When $p=1$, by using the techniques in our article, we can also give an alternative proof of \cite[Thm 5]{Tou16}, cf. Proposition \ref{pless2}. It will imply the following property:

\begin{proposition} \label{proposition-lc}
Let $Y$ be a normal compact K\"ahler space with lc singularities.
Let $\sA$ be a rank one reflexive sheaf such that the reflexive power $\sA^{[m]}$ is locally free and pseudoeffective for some $m \in \N$. Let 
$$
u \in H^0(Y, (\Omega_Y \otimes \sA^\star)^{\star\star})
$$
be a non-zero holomorphic section. Let $S_u \subset T_Y$ be the saturated subsheaf given by vector field $\xi$ such that the contraction $i_\xi u$ vanishes.
Then $S_u$ is integrable, i.e. it defines a (possibly singular) foliation on $Y$.
\end{proposition}

Patrick Graf indicated an alternative path of proof of Proposition \ref{proposition-lc}: by \cite[Thm.1.4]{GK14}\footnote{The statement is formulated for algebraic varieties, but in view of \cite{KS21} should hold for analytic spaces.} a holomorphic $1$-form on the smooth locus of a log-canonical space
extends to a resolution, even without admitting logarithmic poles. Therefore we can copy the proof of Theorem \ref{theorem-main} and verify the technical condition of Proposition \ref{l2version}. Note that \cite[Thm.1.6]{GK14} gives an example of a $2$-form on a $3$-fold  that does not extend to a resolution unless we admit logarithmic poles. Therefore this approach does not allow to generalise Proposition \ref{proposition-lc} to forms in $(\Omega_Y^p \otimes \sA^\star)^{\star\star}$ with $p \geq 2$.

{\bf Acknowledgements.} This work was initiated by our discussions with Mihai P\v aun
during his stay in Nice in the spring 2022 for which we are very grateful.
We thank Patrick Graf and Stefan Kebekus for communications on forms on singular spaces, St\'ephane Druel, Daniel Greb, Wenhao Ou, Mihai P\v aun and Fr\'ed\'eric Touzet for their valuable comments on the article.
The Institut Universitaire de France and A.N.R project Karmapolis (ANR-21-CE40-0010)  provided excellent working conditions for this project.

\section{Notation and terminology}

For general definitions in complex and algebraic geometry we refer to \cite{Har77, Dem12},
for the terminology of singularities of the MMP we refer to \cite{KM98}.
Manifolds and normal complex spaces will always be supposed to be irreducible.

Given a normal complex space $Y$, we denote by $\Omega_Y^{[p]}:= (\Omega_Y^p)^{\star\star}$ the sheaf of holomorphic reflexive $p$-forms. If $Y$ has klt singularities we know by \cite[Thm.1.1]{KS21} that this coincides with the sheaf of holomorphic $p$-forms that extend to a resolution of singularities $X \rightarrow Y$.

For a reflexive sheaf $\sF$ on $Y$, we denote by $\sF^{[m]}:=(\sF^{\otimes m})^{\star\star}$
the $m$-th reflexive power. Given a surjective morphism $\holom{\varphi}{X}{Y}$ we denote by $\varphi^{[\star]} \sF$ the reflexive pull-back $(\varphi^\star \sF)^{\star\star}$.

\section{Twisted logarithmic forms}
	
\begin{proposition} \label{proposition-smooth-metric}
 Let $X$ be a compact K\"ahler manifold, and let $E=\sum E_i$ be a snc divisor. Let $(L,h)$ be a holomorphic line bundle on $X$ where $h$ is a smooth metric such that $i\Theta_h (L)\geq 0$. Let $u\in H^0 (X, \Omega_X ^p (\log E) \otimes L^\star)$ and $(L^\star, h^\star)$ be the dual metric on $(L,h)$. Then $D' _{h^\star} u =0$ on $X$ and $i\Theta_h (L) \wedge u\wedge \overline{u} =0$.
\end{proposition}

\begin{proof}
If $L$ is a trivial line bundle, it is done by \cite{Nog95}. We generalize it to the twisted setting by the following argument.

{\em Step 1:}  Since $h$ is a smooth metric, we know that $D' _{h^\star} u \in C^\infty (X, \Omega^{p+1}_X (\log E) \otimes L^\star)$. We show in this step that 
$D' _{h^\star} u \in C^\infty (X, \Omega^{p+1}_X  \otimes L^\star)$.

We consider the residue  of $u$ and  $D' _{h^\star} u$ on $E_i$. 
First of all, by a direct calculation, we have 
\begin{equation}\label{equality1}
\Res_{E_i} (D' _{h^\star} u) = - D'_{h^\star} \Res_{E_i} (u) \qquad\text{ on } E_i .
\end{equation}	
In fact, let $\Omega$ be a neighborhood of a generic point of $E_i$. We suppose that $E_i$ is defined by $z_1=0$ and  $h =e^{-\varphi}$ on $\Omega$. 
Then we can write
$$u = \frac{dz_1}{z_1} \wedge f +g$$ 
for two smooth forms $f, g$ on $\Omega$. 

For the RHS of \eqref{equality1}, since $\Res_{E_i} (u)= f$ and  we obtain 
$$ - D'_{h^\star} \Res_{E_i} (u) = -(\partial f + \partial \varphi \wedge f) |_{E_i} .$$ 
For the LHS of \eqref{equality1}, we have
$$
\Res_{E_i} (D' _{h^\star} u) = 
\Res_{E_i} (D' _{h^\star} (\frac{dz_1}{z_1} \wedge f)) = \Res_{E_i} ( - \frac{dz_1}{z_1} \wedge \partial f  + \partial \varphi \wedge \frac{dz_1}{z_1} \wedge f)=- (\partial f +\partial \varphi \wedge f )|_{E_i}.
$$
Then we obtain \eqref{equality1}.

Note that $\Res_{E_i} (u) \in H^0 (E_i , \Omega_{E_i} ^{p-1} (\log (E-E_i)) \otimes L^\star)$. By induction on dimension, we know that $\Res_{E_i} (u)$ is $D' _{h^\star}$-closed on $E_i$. Then \eqref{equality1} implies that $\Res_{E_i} (D' _{h^\star} u) =0 $. Therefore the form $D' _{h^\star} u$  is a smooth form on the total space $X$.

\medskip

{\em Step 2:}  Let $N\in\mathbb N^\star$ and let $\Xi_N (x)$ be a smooth function which equals to $1$ on $[0, N]$, equals to $0$ on $[N+1 ,\infty]$ and $0 \leq \Xi'_N (x)\leq 1$. Let $s_E$ be the canonical section of $E$. We consider the integration
\begin{equation}\label{twoterms}
 \int_X \Xi_N (\log (- \log |s_E|)) \{D' _{h^\star} u, D' _{h^\star} u\} \wedge \omega_X ^{n-p-1} .
\end{equation}
Here $|s_E|$ denotes the norm of $s_E$ with respect to a fixed smooth metric on $E$.

By integration by parts, \eqref{twoterms} equals to
 $$= \int_X \{ D' _{h^\star} (\Xi_N (\log (- \log |s_E|)) u) , D' _{h^\star}u\} \wedge \omega_X ^{n-p-1} - \int_X \{\partial (\Xi_N (\log (- \log |s_E|))) \wedge u , D'_{h^\star} u \} \wedge \omega_X ^{n-p-1} $$
$$ = -\int_X  (-1)^p \Xi_N (\log (- \log |s_E|)) \{u, \dbar (D' _{h^\star} u)  \}\wedge \omega_X ^{n-p-1} - \int_X \{\partial (\Xi_N (\log (- \log |s_E|))) \wedge u , D'_{h^\star} u \} \wedge \omega_X ^{n-p-1}$$
\begin{equation}\label{twotermsest}
	=- \int_X i\Theta_h (L) \Xi_N \cdot  \{u, u\}\wedge \omega_X ^{n-p-1} - \int_X \{ \frac{\Xi' _N \cdot \partial \log |s_E| \wedge u }{\log |s_E|} , D'_{h^\star} u \} \wedge \omega_X ^{n-p-1} . 
\end{equation}

Since $i\Theta_h (L)\geq 0$, the first term of \eqref{twotermsest} is semi-negative. 
For the second term of \eqref{twotermsest}, by step 1, we know that  $D'_{h^\star} u$ is smooth on $X$.  Together with $\frac{d s_{E_i}}{ s_{E_i} \log |s_{E_i}|} \wedge \frac{d s_{E_i}}{ s_{E_i}} =0$, we know that the second term of \eqref{twotermsest} is controlled by 
$$\int_{N \leq \log (- \log |s_E|) \leq N+1} \frac{1}{\prod_i |s_{E_i}|} \omega_X ^{n} ,$$ 
which converges to zero when $N \rightarrow 0$.

As a consequence, when $N\rightarrow +\infty$, the upper limit of \eqref{twotermsest} will not be strictly positive. Since  \eqref{twoterms} is always positive, we obtain
\begin{equation}
\lim_{N\rightarrow +\infty}	\int_X \Xi_N (\log (- \log |s_E|)) \{D' _{h^\star} u, D' _{h^\star} u\} \wedge \omega_X ^{n-p} =0 .
\end{equation}
Therefore $D'_{h^\star} u =0$ on $X$.
\end{proof}

\begin{remark}
By a standard argument, it is easy to generalize the above proposition to the case when the metric $(L,h)$ is of analytic singularity. However, it is unclear whether we can generalize it to the case of arbitrary singularity  cf. Question \ref{lccase}.
\end{remark}

In the rest of the section, we will confirm Question \ref{lccase} in two special cases.

\begin{proposition} \label{l2version}
	Let $(X,\omega_X)$ be a compact K\"ahler manifold, and let $E=\sum_{i=1}^r E_i$ be a snc divisor. Let $(L,h)$ be a holomorphic line bundle on $X$ where $h$ is a possibly singular metric such that  $i\Theta_h (L)\geq 0$ on $X$ in the sense of currents. Let $(L^\star, h^\star)$ be the dual metric.  
	Let $u\in H^0 (X, \Omega_X ^p (\log E) \otimes L^\star)$. We assume that $\Res_{E_i} (u) \neq 0$ for every $1 \leq i\leq k$ and $\Res_{E_i} (u)=0$ for every $k<i \leq r$. 
	
We write $h=e^{-\varphi} \cdot h_0$, where $\varphi$ is a quasi-psh function on $X$ and $h_0$ is a smooth metric on $L$. If the weight function $\varphi$ satisfies:
\begin{equation}\label{upper}
	\varphi \leq - 2 \sum_{i=1}^k \ln (-\ln |s_{E_i}|) + C ,
	\end{equation}
where $s_{E_i}$ is the canonical section of $E_i$, then $D' _{h^\star} u =0$ and $i\Theta_h (L) \wedge u \wedge \overline{u} =0$ on $X\setminus E$ , where $D' _{h^\star}$ is the connection with respect to $h^\star$.
	\end{proposition}

\begin{remark}
Note that if the Lelong number of $\varphi$ along $E_i$ is strictly positive for every $i\leq k$, then
$\varphi$ satisfies the condition \eqref{upper}.
\end{remark}

\begin{proof}
The proof is divided into two steps.

	{\em Step 1:}  Let $N\in\mathbb N^\star$ and let $\Xi_N (x)$ be a smooth function which equals to $1$ on $[0, N]$, equals to $0$ on $[N+1 ,\infty]$ and $0 \leq \Xi'_N (x)\leq 1$.  We consider the integration
\begin{equation}\label{twotermssing1}
	\int_X \Xi^2 _N (\log (\log (- \log |s_E|))) \{D' _{h^\star} u, D' _{h^\star} u\} \wedge \omega_X ^{n-2} .
\end{equation}
Since $D' _{h^\star} u$ is $L^2$ in the support of $ \Xi_N (\log (\log (- \log |s_E|))) $, we can still do the integration by parts as in \cite{Dem02}. In particular, \eqref{twotermssing1} equals to
$$= \int_X \{ D' _{h^\star} (\Xi^2 _N (\log (\log (- \log |s_E|))) u) , D' _{h^\star}u\} \wedge \omega_X ^{n-2} - \int_X \{\partial (\Xi^2 _N (\log (\log (- \log |s_E|))) \wedge u , D'_{h^\star} u \} \wedge \omega_X ^{n-2} $$
\begin{equation}\label{twotermsestsing1}
	=- \int_X i\Theta_h (L) \Xi^2 _N (\log (- \log |s_E|))  \{u, u\}\wedge \omega_X ^{n-2} - \int_X \{ \frac{2 \cdot \Xi' _N \cdot \partial \log |s_E| \wedge u }{\log(-\log |s_E|) \log |s_E|}  , \Xi_N \cdot D'_{h^\star} u \} \wedge \omega_X ^{n-2} . 
\end{equation}

Since $i\Theta_h (L)\geq 0$, the first term of \eqref{twotermsestsing1} is semi-negative. 
For the second term of \eqref{twotermsestsing1}, by using Cauchy inequality, we get 
$$| \int_X \{ \frac{\Xi' _N \cdot  \partial \log |s_E| \wedge u}{\log(-\log |s_E|) \log |s_E|}  , \Xi_N \cdot D'_{h^\star} u \} \wedge \omega_X ^{n-2}| ^2 $$
$$\leq 	\int_X \Xi^2 _N \{D' _{h^\star} u, D' _{h^\star} u\} \wedge \omega_X ^{n-2} \cdot 
\int_X \{ \frac{\Xi' _N \cdot   \partial \log |s_E| \wedge u}{\log(-\log |s_E|) \log |s_E|}  , \frac{\Xi' _N \cdot   \partial \log |s_E| \wedge u }{\log(-\log |s_E|) \log |s_E|} \} \wedge \omega_X ^{n-2} .$$
As a consequence, we obtain 
\begin{equation}\label{bound1}
	\int_X \Xi^2 _N \cdot  \{D' _{h^\star} u, D' _{h^\star} u\} \wedge \omega_X ^{n-2} \leq 	\int_X \{ \frac{\Xi' _N \cdot   \partial \log |s_E| \wedge u}{\log(-\log |s_E|) \log |s_E|}  , \frac{\Xi' _N \cdot   \partial \log |s_E| \wedge u}{\log(-\log |s_E|) \log |s_E|} \} \wedge \omega_X ^{n-2}
\end{equation}

\bigskip

{\em Step 2:}  In this step, we would like to show the RHS of \eqref{bound1} tends to zero when $N \rightarrow +\infty$.

Since $\frac{d s_{E_i}}{ s_{E_i}} \wedge \frac{d s_{E_i}}{ s_{E_i}} =0$, the assumption \eqref{upper} implies that 
$\{\partial \log |s_E| \wedge u , \partial \log |s_E| \wedge u \} \wedge \omega_X ^{n-2}$ is upper bounded  by 
$$ C'   \frac{\omega_X ^n }{\prod_{i=1}^k |s_{E_i}|^2 \log^2 |s_{E_i}|}\cdot  (\sum_{i =k+1}^r \frac{1}{|s_{E_i}|^2})$$
for some constant $C'$. 
Then the RHS of \eqref{bound1} is controlled by 
\begin{equation}\label{maintron1}
	C' \sum_{i=k+1}^r\int_X  \frac{(\Xi' _N)^2  \omega_X ^n}{\prod_{i=1}^k |s_{E_i}|^2 \log^2 |s_{E_i}|} \cdot  \frac{1}{|s_{E_i}|^2 \log^2 |s_{E_i}|}.
\end{equation}
which converges to zero when $N \rightarrow 0$.  As a consequence, the RHS of \eqref{bound1} tends to zero when $N \rightarrow +\infty$. Therefore $D' _{h^\star} u =0$ on $X\setminus E$.
	\end{proof}

\noindent By using the argument in Proposition \ref{l2version}, we can give an alternative proof of \cite[Thm 5]{Tou16}:

\begin{proposition} \label{pless2}
	Let $X$ be a compact K\"ahler manifold, and let $E=\sum E_i$ be a snc divisor. Let $(L,h)$ be a holomorphic line bundle on $X$ where $h$ is a possible singular metric such that $i\Theta_h (L)\geq 0$.  Let $u\in H^0 (X, \Omega_X ^1 (\log E) \otimes L^\star)$ and $(L^\star, h^\star)$ be the dual metric on $(L,h)$. 
	Then $D' _{h^\star} u =0$ and $i\Theta_h (L) \wedge u \wedge \overline{u} =0$ on $X\setminus E$.
\end{proposition}

\begin{proof}
	
	We follow the notations in Proposition \ref{l2version}. By the step 1 of Proposition \eqref{bound1}, we know that 
	\begin{equation}\label{bound}
	\int_X \Xi^2 _N \cdot  \{D' _{h^\star} u, D' _{h^\star} u\} \wedge \omega_X ^{n-2} \leq 	\int_X \{ \frac{\Xi' _N \cdot   \partial \log |s_E| \wedge u}{\log(-\log |s_E|) \log |s_E|}  , \frac{\Xi' _N \cdot   \partial \log |s_E| \wedge u}{\log(-\log |s_E|) \log |s_E|} \} \wedge \omega_X ^{n-2}
\end{equation}
In order to prove the proposition, it is sufficient to show the RHS of \eqref{bound} tends to zero when $N \rightarrow +\infty$. 

Since $\frac{d s_{E_i}}{ s_{E_i}} \wedge \frac{d s_{E_i}}{ s_{E_i}} =0$ and $u$ is a $1$-form, 
$\{\partial \log |s_E| \wedge u , \partial \log |s_E| \wedge u \} \wedge \omega_X ^{n-2}$ is upper bounded  by 
$$ C \cdot \sum_{i \neq j} \frac{\omega_X ^n }{ |s_{E_i} s_{E_j}|^2 } .$$
Then the RHS \eqref{bound} is controlled by 
\begin{equation}\label{maintron}C \sum_{i\neq j}\int_X  \frac{(\Xi' _N)^2  \omega_X ^n}{\log^2 (-\log |s_E|) \log^2 |s_E| \cdot  |s_{E_i} s_{E_j}|^2} .
	\end{equation}
Note that the integral
$$\int_{ 0\leq r_1 ,r_2 \leq 1}  \frac{d r_1 \wedge dr_2}{\log^2 (-\log |r_1 r_2|) \log^2 |r_1 r_2| \cdot r_1 r_2} < +\infty .$$
Therefore \eqref{maintron} converges to zero when $N \rightarrow 0$.  As a consequence, the RHS of \eqref{bound} tends to zero when $N \rightarrow +\infty$. Therefore $D' _{h^\star} u =0$ on $X\setminus E$.
\end{proof}

\section{Lifting subsheaves to the resolution}

Let $Y$ be a normal complex space with klt singularities, and let
$\holom{\nu}{Y'}{Y}$ be a proper surjective morphism from a normal complex space $Y'$.  Since klt singularities are rational \cite[Thm.5.22]{KM98}, by \cite[Thm.1.10]{KS21} there exists for every $p \in \N$ a cotangent map
\begin{equation} \label{equation-pull-back}
	d \nu : \nu^\star \Omega_Y^{[p]} \rightarrow \Omega_{Y'}^{[p]}
\end{equation}
If $Y$ has lc singularities we can still combine the proof of \cite[Thm.4.3]{GKKP11}
with \cite[Thm.1.5]{KS21} to obtain\footnote{Note that \cite[Thm.1.10]{KS21} holds for any morphism, while we only need the simpler case where the morphism is surjective.} that 
there exists for every $p \in \N$ a cotangent map
\begin{equation} \label{equation-pull-back-lc}
	d \nu : \nu^\star \Omega_Y^{[p]} \rightarrow \Omega_{Y'}^{[p]}(\log \Delta)
\end{equation}
where $\Delta \subset Y'$ is the largest reduced Weil divisor contained
in $\nu^{-1}(\mbox{non-klt locus})$.

The following statement is well-known to experts and essentially a rewriting of the proof of
\cite[Thm.7.2]{GKKP11}. We include it for the convenience of the reader:

\begin{lemma}\label{lemma-pullback}
	Let $Y$ be a normal complex space with lc singularities, and let $\mathcal A \subset \Omega_Y^{[p]}$ be a reflexive subsheaf of rank one that is $\Q$-Cartier, i.e. there exists a $m \in \N$ such that $\mathcal A^{[m]}$ is locally free.
	
	Let $\holom{\pi}{X}{Y}$ be a log resolution, and let $E$ be the exceptional divisor. Let $\mathcal C \subset \Omega^p_{X}(\log E)$ be the saturation of the image of the morphism
	$$
	\pi^\star \mathcal A \rightarrow \pi^\star \Omega_Y^{[p]} \stackrel{d \pi}{\rightarrow} \Omega^p_{X}(\log E).
	$$
	Then there exists a non-zero morphism $\pi^\star \mathcal A^{[m]} \rightarrow \mathcal C^{\otimes m}$.
\end{lemma}

\begin{remark*}
The morphism $\pi^\star \mathcal A^{[m]} \rightarrow \mathcal C^{\otimes m}$ is an isomorphism in the complement of the exceptional divisor $E$. Thus, up to multiplication by a holomorphic function that is a pull-back from $Y$, the morphism is unique. 
\end{remark*}

If $Y$ has klt singularities, we could use \eqref{equation-pull-back} and consider $\mathcal C' \subset  \Omega^p_{X}$, the saturation 
of the image of the morphism
	$$
	\pi^\star \mathcal A \rightarrow \pi^\star \Omega_Y^{[p]} \stackrel{d \pi}{\rightarrow} \Omega^p_{X},
	$$
but in general there will be no morphism $\pi^\star \mathcal A^{[m]} \rightarrow (\mathcal C')^{\otimes m}$.
However, in the course of the proof of Lemma \ref{lemma-pullback} we will prove the following remark that will be useful for the proof of Lemma \ref{lemma-residue}:

\begin{remark} \label{remark-klt-case}
If $Y$ is klt, let $\holom{\tilde \gamma}{\tilde Z}{X}$ be the cover induced by a (local) index-one cover
$\holom{\gamma}{Z}{Y}$ of $\mathcal A$ (cf. Diagram \eqref{commute}). Then  $\pi_Z^\star \gamma^\star \mathcal A^{[m]}$ is a subsheaf of $S^{[m]}  \Omega_{\tilde Z}^{[p]}$.  
\end{remark}

For the proof let us recall the notion of index one covers \cite[Defn.5.19]{KM98}: given a normal complex space $Y$
and a reflexive sheaf $\mathcal A$ such that some reflexive power $\mathcal A^{[m]}$ is  trivial, there exists a quasi-\'etale morphism $\holom{\gamma}{Z}{Y}$ from a normal complex space $Z$ such that the reflexive pull-back $\gamma^{[\star]} \mathcal A$ is locally free.

\begin{proof}[Proof of Lemma \ref{lemma-pullback}]
	The locally free sheaves coincide in the complement of the exceptional locus $E= \cup_i E_i$,
	so we can write $\mathcal C^{\otimes m} \simeq \pi^\star \mathcal A^{[m]} \otimes \sO_{X}(\sum a_i E_i)$ with  uniquely determined $a_i \in \Z$. We are done if we show that $a_i \geq 0$ for all $i$. This property can be checked locally on the base $Y$.
	
	Therefore we can replace $Y$ by a Stein neighborhood 
	such that there exists an index-one cover $\holom{\gamma}{Z}{Y}$, and let $\holom{\tilde \gamma}{\tilde Z}{X}$ be the induced finite map from the normalisation $\tilde Z$ of $X \times_Y Z$. 
	We denote by $\holom{\pi_Z}{\tilde Z}{Z}$ the bimeromorphic morphism induced by $\pi$ and  summarize the construction in a commutative diagram:
	\begin{equation} \label{commute}
	\xymatrix{
\tilde Z \ar[r]^{\tilde \gamma} \ar[d]_{\pi_Z} & X \ar[d]^\pi
\\
Z \ar[r]^\gamma & Y 
}
	\end{equation}
	
	The morphism $\holom{\gamma}{Z}{Y}$ is an index-one cover for $\mathcal A$,
	so $\gamma$ is \'etale in codimension one and $\gamma^{[\star]} \mathcal A =: \mathcal B$
	is locally free. In particular $Z$ still has lc singularities \cite[Prop.5.20(4)]{KM98}.
		Denote the exceptional locus of $\pi_Z$ by $E_Z$ and observe that $E_Z$ is equal to the support of $\tilde \gamma^\star E$. In particular $E_Z$ contains the preimage of the non-klt locus 
		of $Z$, so 
	\eqref{equation-pull-back-lc} gives a natural map
	$$
	d \pi_Z : \pi_Z^\star \Omega_Z^{[p]} \rightarrow \Omega_{\tilde Z}^{[p]}(\log E_Z)
	$$

	Since $\mathcal A \subset \Omega_Y^{[p]}$ and $\gamma$ is \'etale in codimension one
	we have an inclusion $\mathcal B \subset \Omega_Z^{[p]} \simeq \gamma^{[\star]} \Omega_Y^{[p]}$ and hence
	an induced map
	$$
	\pi_Z^\star \mathcal B \rightarrow \pi_Z^\star  \Omega_Z^{[p]} \rightarrow 
	\Omega_{\tilde Z}^{[p]}(\log E_Z). 
	$$
	Since $\mathcal B$ is locally free, this induces an inclusion
	
	\begin{equation} \label{equation-inclusion-B}
	\pi_Z^\star \mathcal B^{\otimes m} \simeq (\pi_Z^\star \mathcal B)^{\otimes m}
	\rightarrow S^{[m]} \Omega_{\tilde Z}^{[p]}(\log E_Z). 
	\end{equation}
	
	By assumption $A^{[m]}$ is locally free, so its (non-reflexive !) pull-back $\gamma^\star \mathcal A^{[m]}$ is still locally free. Thus $B^{\otimes m} \simeq \gamma^\star A^{[m]}$ since they are both reflexive and coincide in codimension one. Thus we have constructed a morphism
	$$
	\pi_Z^\star \gamma^\star A^{[m]} \rightarrow S^{[m]} \Omega_{\tilde Z}^{[p]}(\log E_Z). 
	$$

We interrupt the proof of the lemma for the {\em Proof of Remark \ref{remark-klt-case}.} 	

If $Y$ is klt, the index one cover $Z$ also has klt singularities \cite[Prop.5.20(4)]{KM98}.
Thus we can replace the pull-back with logarithmic poles \eqref{equation-pull-back-lc} by
the usual pull-back \eqref{equation-pull-back} to obtain
$$
	d \pi_Z : \pi_Z^\star \Omega_Z^{[p]} \rightarrow \Omega_{\tilde Z}^{[p]}  
	$$
As above the inclusion $\gamma^{[\star]} \mathcal A \simeq \mathcal B \subset \Omega_Z^{[p]} \simeq \gamma^{[\star]} \Omega_Y^{[p]}$ then gives the inclusion
$$
	\pi_Z^\star \gamma^\star \mathcal A^{[m]} \simeq  \pi_Z^\star \mathcal B^{\otimes m} \simeq (\pi_Z^\star \mathcal B)^{\otimes m}
	\rightarrow S^{[m]}  \Omega_{\tilde Z}^{[p]}. 
$$
This proves Remark \ref{remark-klt-case}, we now proceed with the proof of Lemma \ref{lemma-pullback}.

	Since $X$ is smooth, the saturated subsheaf $\mathcal C \subset \Omega^p_{X}(\log E)$ is locally free and a subbundle in codimension one. Thus 
	\begin{equation} \label{inclusion-m}
	\mathcal C^{\otimes m} \subset S^m \Omega^p_{X}(\log E)
	\end{equation}
	is locally free and a subbundle in codimension one, hence a saturated subsheaf. The finite morphism $\tilde \gamma$ is \'etale in the complement
	of $E$ and $\Omega^p_{X}(\log E)$ is locally free, so the tangent map gives an isomorphism
	\begin{equation} \label{iso-logarithmic}
	\tilde \gamma^\star \Omega^p_{X}(\log E) \simeq \Omega_{\tilde Z}^{[p]}(\log E_Z).
\end{equation}
and hence an isomorphism
	$$
	\tilde \gamma^\star S^m \Omega^p_{X}(\log E) \simeq S^{[m]} \Omega_{\tilde Z}^{[p]}(\log E_Z).
	$$
	Composing the inclusion \eqref{inclusion-m} with this isomorphism we obtain that 
	$$
	\tilde \gamma^\star \mathcal C^{\otimes m} \rightarrow  S^{[m]} \Omega_{\tilde Z}^{[p]}(\log E_Z)
	$$
	is a saturated subsheaf.
	
	Since $Y$ is Stein and $\mathcal A^{[m]}$ is invertible we can choose for every point $y \in Y$ a section $\sigma \in H^0(Y, \mathcal A^{[m]})$ that does not vanish in $y$.
	In particular $\sigma$ generates $\mathcal A^{[m]}$ as an $\sO_Y$-module near the point $y$.
	Thus it induces a section
	$$
	\pi_Z^\star \gamma^\star \sigma \in H^0(\tilde Z, S^{[m]} \Omega_{\tilde Z}^{[p]}(\log E_Z))
	$$
	that generates the image of $\pi_Z^\star \gamma^\star \mathcal A^{[m]}$.
	The pull-back $\pi^\star \sigma$ defines a meromorphic section of $\mathcal C^{\otimes m}$
	that has poles at most along $E$, thus 
	$\tilde \gamma^\star \pi^\star \sigma$ defines a meromorphic section of $\tilde \gamma^\star \mathcal C^{\otimes m}$ that has poles at most along $E_Z$.
	Since $\tilde \gamma^\star \mathcal C^{\otimes m}$ is saturated in $S^{[m]} \Omega_{\tilde Z}^{[p]}(\log E_Z)$ and 
	$$
	\pi_Z^\star \gamma^\star \sigma = \tilde \gamma^\star \pi^\star \sigma \in H^0(\tilde Z, S^{[m]} \Omega_{\tilde Z}^{[p]}(\log E_Z))
	$$
	has no poles, we see that 
	$$
	\tilde \gamma^\star \pi^\star \sigma \in H^0(\tilde Z, \tilde \gamma^\star \mathcal C^{\otimes m}).
	$$ 
	Thus the local generator of the subsheaf $\pi_Z^\star \gamma^\star \mathcal A^{[m]}$
	lies in $\tilde \gamma^\star \mathcal C^{\otimes m}$ and we have an inclusion
	$$
	\tilde \gamma^\star \pi^\star \mathcal A^{[m]}  \simeq \pi_Z^\star \gamma^\star \mathcal A^{[m]} \hookrightarrow \tilde \gamma^\star \mathcal C^{\otimes m}.
	$$
	Thus we see that
	$$
	\tilde \gamma^\star \sO_{X}(\sum a_i E_i) \simeq
	\tilde \gamma^\star (\mathcal C^{\otimes m} \otimes \pi^\star \mathcal A^{[-m]})
	$$
	is represented by an effective divisor with support in the exceptional locus of $\pi_Z$. Since  $\tilde \gamma^\star (\sum a_i E_i)$ is linearly equivalent to an effective, exceptional divisor and has also support in the exceptional locus of $\pi_Z$, it is  effective.
	Thus we have shown that $a_i \geq 0$ for all $i$.
\end{proof} 

As in immediate application we obtain a variant  
of \cite[Thm.7.2]{GKKP11},\cite[Cor.1.3]{Gra15} for pseudoeffective line bundles.

\begin{cor}\label{corollary-pseff}
	Let $Y$ be a normal compact complex space with lc singularities, and let $\mathcal A \subset \Omega_Y^{[p]}$ be a reflexive subsheaf of rank one that is $\Q$-Cartier, i.e. there exists a $m \in \N$ such that $\mathcal A^{[m]}$ is locally free. Let  $\mathcal C \subset \Omega^p_{X}(\log E)$ be the saturation of $\pi^\star
\sA$. If $\mathcal A^{[m]}$ is pseudoeffective, then $\mathcal C$ is pseudoeffective.
\end{cor}

\begin{proof} Since pseudoeffectivity of a line bundle is invariant under taking
tensor powers, it is sufficient to show that $\mathcal C^{\otimes m}$ is pseudoeffective. Yet this follows from the non-zero morphism
$\pi^* \mathcal A^{[m]} \rightarrow \mathcal C^{\otimes m}$ constructed in 
Lemma \ref{lemma-pullback}.
\end{proof}

We need the following proposition.

\begin{lemma} \label{lemma-residue}
In the situation of Lemma \ref{lemma-pullback}, write
\begin{equation}\label{equality}
	\mathcal{C}^{\otimes m} = \pi^\star \mathcal{A}^{[m]} \otimes \sO_X(\sum a_i E_i),
	\end{equation}
where $a_i \geq 0$ and $E=\sum E_i$ is the exceptional locus.  

Assume that $Y$ has klt singularities, and let $E_i$ be an irreducible component of the exceptional locus. 
Let $\Res_{E_i} (\mathcal{C})$ be the residue of the image of $\mathcal{C}$ in $\Omega^p_{X}(\log E)$. If $\Res_{E_i} (\mathcal{C}) \neq 0$, then $a_i >0$. 
\end{lemma}	

\begin{proof}
The claim is local on $Y$, so we will use the construction from the proof of Lemma \ref{lemma-pullback} summarized in the commutative diagram \eqref{commute}.

Fix a prime divisor $\tilde E_i \subset \tilde Z$ that maps onto $E_i \subset X$,
and choose a general point $\tilde x \in \tilde E_i \cap \tilde Z_{\nons}$ such that $\tilde E_i$ (resp. $E_i$)
is smooth in $\tilde x$ (resp. smooth in $x:=\tilde \gamma(\tilde x)$). Since $\tilde x$ is general, the finite morphism $\tilde \gamma$ has constant rank in an analytic neighborhood of $\tilde \gamma$, hence we can find local coordinates on $\tilde Z$ and $X$ such that 
$$
E_i = \{ z_1=0 \}
$$ 
and $\tilde \gamma$ is given locally by
$$
\widetilde{\gamma}: (t, z_2, .. z_n )
\rightarrow 
(t^d, z_2, .., z_n ).
$$
The exterior power $\Omega^p _{X} (\log E)_x$ is generated by
$\{\frac{dz_1}{z_1} \wedge d z_J, d z_I \}$
where $J \subset \{2,\cdots, n\}$ has length $p-1$ and $I \subset \{2,\cdots, n\}$ has length $p$.
Thus we obtain a basis $\{e_1, \cdots, e_k\}$ of $S^m \Omega_{X}  (\log E)_x$ by taking products of length $m$, where each $e_i$ is of type:
$$e_i = (\frac{dz_1}{z_1} \wedge d z_{J_1}) \otimes (\frac{dz_1}{z_1} \wedge d z_{J_2})\otimes \dots \otimes (\frac{dz_1}{z_1} \wedge d z_{J_q}) \otimes d z_{I_1 }\otimes \dots \otimes dz_{I_{m-q}}. $$

In our local coordinates the pull-back becomes
$$\widetilde{\gamma}^\star (e_i)
= (\frac{dt}{t}\wedge d z_{J_1}) \otimes (\frac{dt}{t} \wedge d z_{J_2 })\otimes \dots \otimes (\frac{dt }{t} \wedge d z_{J_q}) \otimes d z_{I_1 }\otimes \dots \otimes dz_{I_{m-q}} .$$
In particular, the pull back $\{\widetilde{\gamma}^\star (e_i)\}_{i=1}^k$ is a basis of $S^m \Omega_{\widetilde{Z}}  (\log E_Z)$ at $\tilde x$.

Let $\sigma$ be a generator of $\mathcal{A}^{[m]}$ at $\pi (x)\in Y$. Then $\pi^\star\sigma \in \pi^\star \mathcal{A}^{[m]} \subset S^m \Omega_{X}(\log E)$ is a local generator near $x$. 
We can write 
$$\pi^\star\sigma = \sum f_i e_i ,$$
where $f_i$ are holomorphic functions near $x$.
Now recall that by Remark \ref{remark-klt-case} 
$$
\pi_Z^\star \mathcal B^{\otimes m} \simeq \pi_Z^\star \gamma^\star \mathcal A^{[m]}
\simeq \tilde \gamma^\star \pi^\star \mathcal A^{[m]}
$$
is a subsheaf of $S^{[m]} \Omega^{[p]}_{\widetilde{Z}}$.
In particular, since $\tilde Z$ is smooth in $\tilde x$, we have
$$
(\widetilde{\gamma}\circ \pi)^\star \sigma \in (S^m \Omega^p_{\widetilde{Z}})_{\tilde x}.
$$
As a consequence, $f_i (x)=0$ when $e_i$ is of type
$$
e_i = (\frac{dz_1}{z_1} \wedge d z_{J_1}) \otimes (\frac{dz_1}{z_1} \wedge d z_{J_2})\otimes \dots \otimes (\frac{dz_1}{z_1} \wedge d z_{J_m}),
$$
since this generator of $(S^m \Omega^p_{\widetilde{Z}}(\log E_Z))_{\tilde x}$ is not 
contained in $(S^m \Omega^p_{\widetilde{Z}})_{\tilde x}$.

\medskip

Now we can prove the proposition. Near a general point $x\in E_i$, we suppose that $\mathcal{C}_x \subset (\Omega^p_{\widetilde{Z}})_{\tilde x}$ is generated by 
$$\sum g_i  \cdot (\frac{dz_1}{z_1} \wedge d z_{J_i}) + \sum h_i \cdot d z_{I_i} ,$$
where $g_i, h_i$ are holomorphic functions. 
Thanks to Lemma \ref{lemma-pullback}, we have
$$F \cdot (\sum g_i  (\frac{dz_1}{z_1} \wedge d z_{J_i}) + \sum h_i d z_{I_i} )^{\otimes m} =   (\sum f_i e_i ) ,$$
where $F$ is a holomorphic function near $x$. 
If $\Res_{E_i} (\mathcal{C}) \neq 0$, we know that there is one  $i_0$ such that $g_{i_0} (x) \neq 0$. 
Set 
$$e_{i_0} :=(\frac{dz_1}{z_1} \wedge d z_{J_{i_0}})^{\otimes m} .$$
Then $ F \cdot g_{i_0} ^m =f_{i_0}$. By the above paragraph, we know that $f_{i_0} (x)=0$.
Then $F (x)=0$.
The proposition is thus proved.
\end{proof}

We are now in the position to verify
the technical condition in Proposition \ref{l2version}:

\begin{thm} \label{theorem-conclusion}
In the setting of Theorem \ref{theorem-main}, let $\holom{\pi}{X}{Y}$ be a log-resolution
and denote by $E$ the exceptional locus. Let $L \subset \Omega_{X}^p(\log E)$
be the saturation of $\pi^\star \sA$, and let $\tilde u \in  H^0 (X, \Omega_{X} ^p (\log E) \otimes L^\star)$ the corresponding section.
Then there exists a metric $h_1$ on $L$ such that
we have $D'_{h_1 ^\star} \tilde u =0$ on $X\setminus E$
\end{thm}

\begin{proof}
By Lemma \ref{lemma-pullback}, we know that 
	\begin{equation}\label{inclusion2}
		c_1(L) = \frac{1}{m} \pi^\star c_1(\sA^{[m]}) + \sum_{i\in I} a_i E_i  +\sum_{i\in I' } a_i E_i,
	\end{equation}
	such that all the coefficients $a_i \geq 0$ and the
$i\in I$ correspond to the exceptional divisors $E_i$ such that $\Res_{E_i} (\mathcal{C}) \neq 0$ and $i\in I'$ corresponds to $\Res_{E_i} (\mathcal{C}) = 0$.  By Lemma \ref{lemma-residue} we have $a_i >0$ when $i\in I$.
Let $h_0$ be a possibly singular metric on $\pi^\star \sA ^{[m]}$ such that
$i\Theta_{h_0} (\pi^\star \sA^{[m]})\geq 0$. By \eqref{inclusion2}
this induces a metric $h_1$ on $L$.
Thanks to Proposition \ref{l2version},  the theorem is proved.
\end{proof}

\section{Proof of the main results}

The setup for the proof of Theorem \ref{theorem-main} and Proposition \ref{proposition-lc} is the same:
the non-zero section $u$ determines an injective morphism of sheaves
$$
\sA \hookrightarrow \Omega_Y^{[p]}.
$$
Let $\holom{\pi}{X}{Y}$ be a log-resolution of $Y$, and denote by $E$ the exceptional locus. Since $Y$ is lc, we have the tangent map \eqref{equation-pull-back-lc}
$$
\holom{d \pi}{\pi^\star \Omega_Y^{[p]}}{ \Omega_{X}^{p}(\log E)},
$$
and we denote by $L \subset \Omega^{p}_{X}(\log E)$ the saturation of $\pi^\star \sA$.
By Lemma \ref{lemma-pullback} there exists a morphism
$\pi^\star \sA^{[m]} \rightarrow L^{\otimes m}$, so $L$ is a pseudoeffective line bundle on $X$.
The inclusion $L \subset \Omega^{p}_{X}(\log E)$ corresponds to a non-zero
holomorphic section
$$
\tilde u \in H^0(X, \Omega_{X}^{p}(\log E) \otimes L^\star)
$$
which coincides with $u$ on $X \setminus E \simeq Y_{\nons}$. In particular
the subsheaf $S_{\tilde u} \subset T_{X}$ defined by contraction with $\tilde u$ coincides with $S_u \subset T_Y$ on a Zariski open set.  
Thus we are left to show the integrability of $S_{\tilde u} \subset T_{X}$
on $X \setminus E$. 
By the formula for  the exterior derivative of $p$-forms (cf. \cite[p.97]{Dem02})
the integrability of $S_{\tilde u}$ follows if we find a metric $h$ on $L$ such that
$D'_{h^\star} \tilde u=0$ on $X \setminus E$.

{\em Proof of Theorem \ref{theorem-main}:} Since $Y$ is klt, the existence of the metric $h$
is guaranteed  by Theorem \ref{theorem-conclusion}.
{\hfill $\square$}

{\em Proof of Proposition \ref{proposition-lc}:}
Since $p=1$ we know by Proposition \ref{pless2} that any singular metric with positive curvature current will suffice. Since $L$ is pseudoeffective, such a metric exists.
{\hfill $\square$}


\end{document}